\documentclass[a4paper,oneside,reqno]{amsart}
\usepackage[margin=30mm]{geometry}
\usepackage{mathtools,amssymb}
\usepackage[foot]{amsaddr}
\usepackage[shortlabels]{enumitem}

\usepackage[
 bookmarksnumbered=true,
 colorlinks=true,
 allcolors=blue
]{hyperref}

\usepackage[
 backend=biber,
 style=alphabetic,
 sorting=nyvt
]{biblatex}
\usepackage{xurl}

\numberwithin{equation}{section}

\theoremstyle{plain}
\newtheorem{theorem}{Theorem}[section]
\newtheorem{Itheorem}{Theorem}
\newtheorem{Icorollary}[Itheorem]{Corollary}

\newtheorem{lemma}[theorem]{Lemma}
\newtheorem{proposition}[theorem]{Proposition}
\newtheorem{corollary}[theorem]{Corollary}

\theoremstyle{definition}
\newtheorem{example}[theorem]{Example}
\newtheorem{definition}[theorem]{Definition}
\newtheorem{remark}[theorem]{Remark}

\DeclareMathOperator{\Hom}{Hom}
\DeclareMathOperator{\Gal}{Gal}
\newcommand{\Abelianization}[1]{#1^{\mathrm{ab}}}
\newcommand{\AbsGalGrp}[1]{G_{#1}}
\newcommand{\resfield}[1]{\kappa({#1})}

\begin{document}

\title{Ab-faithfulness of absolute Galois groups in number theory}
\date{\today}

\author[N.~Yamaguchi]{Naganori~Yamaguchi}
\address{Tokyo Denki University, 5 Senju-Asahi-cho, Adachi-ku, Tokyo 120-8551, Japan}
\email{n.yamaguchi@mail.dendai.ac.jp}

\subjclass[2020]{Primary 12F10, 12J10; Secondary 20E18, 20J06, 11S20}
\keywords{ab-faithfulness; absolute Galois groups; discretely valued fields; maximal pro-$\ell$ Galois groups; higher local fields; Kummer theory}

\begin{abstract}
	We prove that the absolute Galois group of every discretely valued field of mixed characteristic is ab-faithful.
	We also prove ab-faithfulness for higher local fields of positive dimension and give further examples using finite split embedding problems.
	Finally, we give a formal Hilbert~90 criterion for center-freeness and slimness of finite-step solvable quotients.
	We apply this criterion to ab-faithful absolute Galois groups and obtain slimness for several classes of fields, including higher local fields of positive dimension.
\end{abstract}

\maketitle
\tableofcontents

\section*{Introduction}

Let $G$ be a profinite group.
We say that $G$ is \emph{ab-faithful} if, for every open subgroup $H$ of $G$ and every open normal subgroup $N$ of $H$, the conjugation action of $H/N$ on $\Abelianization{N}$ is faithful (see~\cite[Definition~1(ii)]{Mochizuki2009CommentsStronglyTorsionFree}).
Here $\Abelianization{N}$ denotes the topological abelianization of $N$.
The following groups are known to be ab-faithful \cite[Examples~5, 6(i), and~7, respectively]{Mochizuki2009CommentsStronglyTorsionFree}:
\begin{enumerate}[(a)]
	\item The absolute Galois groups of finite extensions of $\mathbb{Q}_p$.
	\item Free profinite groups of finite rank at least two and profinite surface groups of genus at least two, as well as their maximal pro-$\ell$ quotients for any prime $\ell$.
	\item The absolute Galois groups of function fields of normal proper varieties over an algebraically closed field of characteristic zero, and their maximal pro-$\ell$ quotients for any prime $\ell$.
\end{enumerate}
In \cite[Proposition~2.8]{yamaguchi2026centerfreenessfinitestepsolvablegroups}, the author also proved that the \'etale and tame fundamental groups of hyperbolic curves over separably closed fields are ab-faithful.

In this paper, we obtain further examples by two methods.
We first give a criterion for ab-faithfulness in terms of finite split embedding problems in Proposition~\ref{prop:elementary-split-criterion} and apply it in Example~\ref{Iexampleemb}.
Next, we prove ab-faithfulness for every discretely valued field of mixed characteristic in Section~\ref{sec:valued-fields}.
We apply this result to higher local fields in Section~\ref{sec:higher-local-fields}.
These results give the following theorem.

\begin{Itheorem}[{Proposition~\ref{prop:elementary-split-criterion} and Theorems~\ref{thm:mixed-characteristic} and~\ref{thm:main-hlf}}]
	\label{intro:main}
	Let $K$ be a field satisfying one of the following conditions:
	\begin{enumerate}[(a)]
		\item\label{item:intro-emb} For some prime $\ell$, every finite split embedding problem for $\AbsGalGrp{K}$ with elementary abelian $\ell$-kernel has a proper solution.
		\item\label{item:intro-mixed-characteristic} $K$ is a discretely valued field of mixed characteristic.
		\item\label{item:intro-hlf} $K$ is a higher local field of positive dimension.
	\end{enumerate}
	Then $\AbsGalGrp{K}$ is ab-faithful.
\end{Itheorem}

A profinite group is \emph{slim} if the centralizer of every open subgroup is trivial.
Ab-faithfulness implies slimness (see~\cite[Remark~1.2]{Mochizuki2009CommentsStronglyTorsionFree}).
The converse fails even for absolute Galois groups; see Lemma~\ref{ex:slim-perfect-kernel}.
For Henselian discretely valued fields of positive characteristic, for discretely valued fields of mixed characteristic, and for higher local fields of positive dimension, Minamide--Tsujimura proved slimness in \cite[Theorems~2.8(i) and~2.10, and Corollary~2.11]{MR4434497}.
Our results prove ab-faithfulness for these fields.

\medskip

Next, we consider finite-step solvable quotients.
For a profinite group $G$, put
\begin{equation*}
	G^{[0]}\coloneq G,\qquad G^{[r+1]}\coloneq [G^{[r]},G^{[r]}],\qquad G^m\coloneq G/G^{[m]},
\end{equation*}
where all commutator subgroups are closed.
In Section~\ref{sec:split-embedding-problems-slim}, we prove that $G^m$ is slim for every $m\ge2$ if every finite split embedding problem for $G$ with abelian kernel has a proper solution.
In particular, this applies to the absolute Galois groups of Hilbertian fields.

In \cite[Proposition~2.5]{yamaguchi2026centerfreenessfinitestepsolvablegroups}, the author proved that $G^m$ is torsion-free and center-free for every $m\geq2$ if $G$ is ab-torsion-free and ab-faithful.
Here \emph{ab-torsion-free} means that the abelianization of every open subgroup is torsion-free.
For a prime $\ell$, an \emph{$\ell$-orientation} is a continuous homomorphism $\theta:G\to\mathbb{Z}_\ell^\times$.
Write $T$ for $\mathbb{Z}_\ell$ with the action $g\cdot t=\theta(g)t$.
The pair $(G,\theta)$ satisfies \emph{formal Hilbert~90} if reduction induces a surjection
\begin{equation*}
	H^1(H,T/\ell^nT)\to H^1(H,T/\ell T)
\end{equation*}
for every open subgroup $H$ of $G$ and every $n\ge1$ (see~\cite[Section~2.2]{MR4880790}).
Every ab-torsion-free group satisfies formal Hilbert~90 for the trivial orientation at every prime.
In Section~\ref{sec:solvable-centers}, we prove the assertion on center-freeness under the formal Hilbert~90 assumption in place of ab-torsion-freeness and give sufficient conditions for slimness.

\begin{Itheorem}[{Corollary~\ref{cor:cf-group-criterion}}]
	\label{intro:solvable-quotients}
	Let $G$ be an ab-faithful profinite group such that, for every prime $\ell$, there exists an $\ell$-orientation of $G$ satisfying formal Hilbert~90.
	Then $G^m$ is center-free for every $m\ge2$.
	If, for every prime $\ell$, such an orientation can be chosen with infinite image, then $G^m$ is slim for every $m\ge2$.
\end{Itheorem}

The following corollary applies Theorem~\ref{intro:solvable-quotients} to absolute Galois groups.

\begin{Icorollary}[{Corollary~\ref{thm:cf-absolute-center}}]
	\label{intro:solvable-galois}
	Let $K$ be a field whose absolute Galois group $\AbsGalGrp{K}$ is ab-faithful.
	Then $\AbsGalGrp{K}^m$ is center-free for every $m\geq2$.
	Assume moreover that the following two conditions hold:
	\begin{enumerate}[(i)]
		\item For every prime $\ell\ne\operatorname{char}K$, the extension $K(\mu_{\ell^\infty})/K$ is infinite.
		\item If $\operatorname{char}K=p>0$, then $K$ admits a cyclic Galois extension of degree $p$.
	\end{enumerate}
	Then $\AbsGalGrp{K}^m$ is slim for every $m\geq2$.
\end{Icorollary}

Thus $\AbsGalGrp{K}^m$ is center-free for every field $K$ in Theorem~\ref{intro:main} and every $m\ge2$.
For mixed-characteristic local fields, center-freeness of $\AbsGalGrp{K}^m$ is proved in \cite[Proposition~A.1]{MR5005799}, where the remark following that proposition attributes the metabelian case to S.~Ladkani.

The ab-faithfulness assumption in Corollary~\ref{intro:solvable-galois} cannot be replaced by slimness: Lemma~\ref{ex:slim-perfect-kernel} gives a slim absolute Galois group whose finite-step solvable quotients are all isomorphic to $\widehat{\mathbb{Z}}$.

Combining the result from finite split embedding problems with the applications of Corollary~\ref{intro:solvable-galois}, we obtain the following corollary.

\begin{Icorollary}[{Proposition~\ref{prop:abelian-split-solvable-slim} and Example~\ref{ex:cf-field-slim}}]
	\label{intro:solvable-galois-2}
	Let $K$ be a field satisfying one of the following conditions:
	\begin{enumerate}[(a)]
		\item\label{item:intro-slim-finitely-generated} $K$ is a Hilbertian field.
		\item\label{item:intro-slim-toral-kummer-faithful} $K$ is a sub-$p$-adic field for a prime $p$.
		\item\label{item:intro-slim-hlf} $K$ is a higher local field of positive dimension.
	\end{enumerate}
	Then $\AbsGalGrp{K}^m$ is slim for every $m\geq2$.
\end{Icorollary}

\section{Ab-faithfulness and the prime-order test}
\label{sec:group-theoretic-criteria}

All homomorphisms of profinite groups are assumed to be continuous, and we take the closure of every commutator subgroup.
For a field $K$, fix a separable closure $K^{\mathrm{sep}}$ and write $\AbsGalGrp{K}\coloneq\Gal(K^{\mathrm{sep}}/K)$; an algebraic closure is denoted by $K^{\mathrm{alg}}$.
Let $\Sigma$ be a set of primes.
A finite group is a \emph{$\Sigma$-group} if every prime divisor of its order belongs to $\Sigma$.
We write $G^{\Sigma}$ for the maximal pro-$\Sigma$ quotient of a profinite group $G$.
For $\Sigma=\{\ell\}$, we use the notation $G(\ell)\coloneq G^{\{\ell\}}$.
We use continuous cohomology, with trivial action on the constant coefficient groups $\mathbb{Z}/n\mathbb{Z}$ unless otherwise specified.
For an abelian profinite group $A$, written additively, and an integer $n\ge1$, write $A/n\coloneq A/nA$.
Thus $\Abelianization{G}/n$ is the maximal abelian quotient of $G$ of exponent dividing $n$.

\begin{definition}
	Let $G$ be a profinite group.
	We say that $G$ is \emph{$\Sigma$-ab-faithful} if, for every open subgroup $H$ of $G$ and every open normal subgroup $N$ of $H$, the conjugation action of $H/N$ on $\Abelianization{(N^{\Sigma})}$ is faithful.
	For $\Sigma=\{\ell\}$, we also say \emph{$\ell$-ab-faithful}.
	For the set of all primes, we simply say \emph{ab-faithful}.
\end{definition}

Let $G$ be a group acting on groups $A$, $B$, and $C$ by automorphisms.
Assume that there are $G$-equivariant homomorphisms
\begin{equation*}
	A\hookleftarrow B\twoheadrightarrow C.
\end{equation*}
Then
\begin{equation*}
	\ker\bigl(G\to\operatorname{Aut}(A)\bigr)
	\subseteq\ker\bigl(G\to\operatorname{Aut}(B)\bigr)
	\subseteq\ker\bigl(G\to\operatorname{Aut}(C)\bigr).
\end{equation*}
In particular, if the action on $C$ is nontrivial (resp.\ faithful), then the action on $A$ is nontrivial (resp.\ faithful).
If the action on $A$ is trivial, then the action on $C$ is trivial.
We use this fact below without further mention.

\begin{lemma}[The prime-order test]
	\label{lem:prime-order-test}
	Let $G$ be a profinite group and $\Sigma$ a set of primes.
	Then the following conditions are equivalent:
	\begin{enumerate}[(i)]
		\item\label{item:prime-order-ab-faithful} The group $G$ is $\Sigma$-ab-faithful.
		\item\label{item:prime-order-abelianization} For every open subgroup $H$ of $G$ and every open normal subgroup $N$ of $H$ such that $H/N$ has prime order, the conjugation action of $H/N$ on $\Abelianization{(N^\Sigma)}$ is nontrivial.
		\item\label{item:prime-order-finite-quotient} For every pair $H,N$ as in \ref{item:prime-order-abelianization}, there exist $\ell\in\Sigma$ and an integer $s\ge1$ such that the conjugation action of $H/N$ on $\Abelianization{N}/\ell^s$ is nontrivial.
		\item\label{item:prime-order-cohomology} For every pair $H,N$ as in \ref{item:prime-order-abelianization}, there exist $\ell\in\Sigma$ and an integer $s\ge1$ such that the induced action of $H/N$ on $H^1(N,\mathbb{Z}/\ell^s\mathbb{Z})$ is nontrivial.
	\end{enumerate}
\end{lemma}

\begin{proof}
	Condition~\ref{item:prime-order-ab-faithful} implies \ref{item:prime-order-abelianization}.
	Conversely, if the kernel of the action of $H/N$ on $\Abelianization{(N^\Sigma)}$ is nontrivial, it contains a subgroup of prime order.
	Replacing $H$ by the inverse image of this subgroup contradicts \ref{item:prime-order-abelianization}.

	By the universal properties of abelianization and maximal pro-$\Sigma$ quotients, and by the primary decomposition of abelian profinite groups, there is a natural $H/N$-equivariant isomorphism
	\begin{equation*}
		\Abelianization{(N^\Sigma)}
		\simeq\prod_{\ell\in\Sigma}\varprojlim_s \Abelianization{N}/\ell^s.
	\end{equation*}
	Thus an automorphism of the left-hand side is trivial if and only if it is trivial on every quotient $\Abelianization{N}/\ell^s$.
	This proves the equivalence of \ref{item:prime-order-abelianization} and \ref{item:prime-order-finite-quotient}.

	For every profinite group $N$ and every integer $n\ge2$, there is a natural identification
	\begin{equation}
		\label{eq:finite-cohomology-duality}
		H^1(N,\mathbb{Z}/n\mathbb{Z})
		=\Hom_{\mathrm{cont}}(\Abelianization{N}/n,\mathbb{Z}/n\mathbb{Z}).
	\end{equation}
	Since these homomorphisms separate points, the actions on $\Abelianization{N}/n$ and $H^1(N,\mathbb{Z}/n\mathbb{Z})$ have the same kernel.
	Taking $n=\ell^s$ proves the equivalence of \ref{item:prime-order-finite-quotient} and \ref{item:prime-order-cohomology}.
\end{proof}

\begin{remark}
	\label{prop:sigma-ab-faithful-properties}
	If $\Sigma\subseteq\Sigma'$, every $\Sigma$-ab-faithful group is $\Sigma'$-ab-faithful.
	If $G$ is $\Sigma$-ab-faithful, then $G^\Sigma$ is $\Sigma$-ab-faithful.
	In particular, if $G$ is $\ell$-ab-faithful, then both $G$ and $G(\ell)$ are ab-faithful.
\end{remark}

\section{Ab-faithfulness and finite split embedding problems}
\label{sec:split-embedding-problems}

We give a sufficient condition for $\ell$-ab-faithfulness in terms of finite split embedding problems and apply it in Example~\ref{Iexampleemb}.
A finite embedding problem for a profinite group $G$ consists of surjections $\pi\colon G\twoheadrightarrow A$ and $\alpha\colon B\twoheadrightarrow A$, where $A$ and $B$ are finite.
A \emph{proper solution} is a surjection $\theta\colon G\twoheadrightarrow B$ satisfying $\alpha\circ\theta=\pi$.
An embedding problem is \emph{split} if $\alpha$ has a group-theoretic section.

An \emph{elementary abelian $\ell$-group} is a finite-dimensional $(\mathbb{Z}/\ell\mathbb{Z})$-vector space, regarded as an additive group.

\begin{proposition}
	\label{prop:elementary-split-criterion}
	Let $G$ be a profinite group and $\ell$ a prime.
	Assume that every finite split embedding problem for $G$ with elementary abelian $\ell$-kernel has a proper solution.
	Then $G$ is $\ell$-ab-faithful.
\end{proposition}

\begin{proof}
	Let $H$ be an open subgroup of $G$ and let $N$ be an open normal subgroup of $H$.
	Put $V\coloneq\bigcap_{g\in G}gNg^{-1}$.
	The hypothesis gives a proper solution
	\begin{equation*}
		\theta:G\twoheadrightarrow(\mathbb{Z}/\ell\mathbb{Z})[G/V]\rtimes(G/V)
	\end{equation*}
	over the natural map $G\twoheadrightarrow G/V$.
	The linear map
	\begin{equation*}
		r:(\mathbb{Z}/\ell\mathbb{Z})[G/V]\twoheadrightarrow(\mathbb{Z}/\ell\mathbb{Z})[H/N],
		\qquad
		[gV]\longmapsto
		\begin{cases}
			[gN], & g\in H,    \\
			0,    & g\notin H,
		\end{cases}
	\end{equation*}
	is well-defined, surjective, and $H$-equivariant.
	Since $\theta(H)=(\mathbb{Z}/\ell\mathbb{Z})[G/V]\rtimes(H/V)$, composing $\theta|_H$ with $(a,hV)\mapsto(r(a),hN)$ gives a surjection
	\begin{equation*}
		H\twoheadrightarrow(\mathbb{Z}/\ell\mathbb{Z})[H/N]\rtimes(H/N)
	\end{equation*}
	over $H\twoheadrightarrow H/N$.
	It maps $N$ onto $(\mathbb{Z}/\ell\mathbb{Z})[H/N]$ and therefore induces the equivariant surjection $\Abelianization{N}/\ell\twoheadrightarrow(\mathbb{Z}/\ell\mathbb{Z})[H/N]$.
	The regular action is faithful.
	Therefore, the surjections
	\begin{equation*}
		\Abelianization{(N(\ell))}\twoheadrightarrow \Abelianization{N}/\ell
		\twoheadrightarrow(\mathbb{Z}/\ell\mathbb{Z})[H/N]
	\end{equation*}
	prove $\ell$-ab-faithfulness.
\end{proof}

\begin{example}\label{Iexampleemb}
	Let $\ell$ be a prime.
	For each field $K$ listed below, every finite split embedding problem for $\AbsGalGrp{K}$ with elementary abelian $\ell$-kernel has a proper solution.
	In particular, these groups are $\ell$-ab-faithful by Proposition~\ref{prop:elementary-split-criterion}.
	\begin{enumerate}[(a)]
		\item\label{item:emb-hilbertian} $K$ is a Hilbertian field, by Ikeda's theorem \cite{MR121364} (see also~\cite[Proposition~18.4.5]{MR4647281}).
		\item $K=\operatorname{Frac}(k[[x,y]])$ for any field $k$, by \cite[Theorem~7.1]{MR2677893}.
		\item\label{item:emb-positive-characteristic} $K$ is a field of characteristic $\ell$ with a valuation whose value group is not $\ell$-divisible, by \cite[Theorem~1.1]{MR2793956}.
	\end{enumerate}
\end{example}

\section{Ab-faithfulness for discretely valued fields of mixed characteristic}
\label{sec:valued-fields}

In this section, we prove $p$-ab-faithfulness for every discretely valued field of mixed characteristic $(0,p)$ (Theorem~\ref{thm:mixed-characteristic}).
We study the actions on first cohomology and use~\eqref{eq:finite-cohomology-duality} to prove faithfulness of the actions on abelian Galois quotients.
For a field $K$, write $\mu_n\coloneq \mu_{n}(K^{\mathrm{sep}})$ for the $n$-th roots of unity in $K^{\mathrm{sep}}$, and $\mu_{\ell^\infty}(K)$ for the $\ell$-power roots of unity contained in $K$.

Let $K$ be a discretely valued field.
The valuation is always nontrivial and normalized by $v_K(K^\times)=\mathbb{Z}$.
Write
\begin{equation*}
	\mathcal{O}_K,\qquad \mathfrak{m}_K,\qquad
	\resfield{K}\coloneq\mathcal{O}_K/\mathfrak{m}_K,\qquad U_K\coloneq1+\mathfrak{m}_K,
\end{equation*}
for the valuation ring, maximal ideal, residue field, and principal units, and write $\widehat K$ for the completion of $K$.
The valued field $K$ has \emph{mixed characteristic $(0,p)$} if $\operatorname{char}K=0$ and $\operatorname{char}\resfield{K}=p>0$.
A valued field is \emph{Henselian} if its valuation extends uniquely to algebraic extensions.
Complete discretely valued fields are Henselian, and their finite extensions are complete for the unique extension valuation; see \cite[Chapter~II, Theorems~4.8 and~6.6]{MR1697859}.

\begin{lemma}
	\label{lem:roots-of-unity}
	Let $K$ be a discretely valued field of mixed characteristic $(0,p)$.
	Then $\mu_{p^\infty}(K)$ is finite.
\end{lemma}

\begin{proof}
	For a primitive $p^n$-th root of unity $\zeta\in K$, the cyclotomic identity $p=\varepsilon(1-\zeta)^{\varphi(p^n)}$, with $\varepsilon$ an algebraic unit, follows from \cite[Chapter~I, Lemma~10.1]{MR1697859}.
	Since $v_K(\varepsilon)=0$, taking valuations gives
	\begin{equation*}
		v_K(p)=\varphi(p^n)v_K(1-\zeta)\ge\varphi(p^n),
	\end{equation*}
	since $v_K(1-\zeta)$ is a positive integer.
	Hence $n$ is bounded.
\end{proof}

We use the following form of Kummer theory.
Let $L$ be a field of characteristic zero, and let $n\ge2$ be an integer such that $\mu_n\subset L$.
We use $n=p$ in this section and $n=\ell^s$ in Lemma~\ref{lem:residue-zero-lift}.
Since $\AbsGalGrp{L}$ acts trivially on $\mu_n$, Kummer theory \cite[Chapter~IV, \S3, Exercise~8]{MR1697859} gives natural isomorphisms
\begin{equation}
	\label{eq:kummer-cohomology}
	H^1(\AbsGalGrp{L},\mathbb{Z}/n\mathbb{Z})\otimes_{\mathbb{Z}/n\mathbb{Z}}\mu_n
	\xrightarrow{\sim}H^1(\AbsGalGrp{L},\mu_n)
	\xrightarrow{\sim}L^\times/L^{\times n}.
\end{equation}
These isomorphisms are equivariant for field automorphisms, acting on both tensor factors.
In particular, their action on $\mu_n$ need not be trivial.
We write Kummer classes additively: $[xy]=[x]+[y]$ and $[x^a]=a[x]$.
Let $\sigma$ be an automorphism of $L$ that fixes $\mu_n$.
Since $\mu_n$ is free of rank one over $\mathbb{Z}/n\mathbb{Z}$ and $\sigma$ acts trivially on it, tensoring with $\mu_n$ preserves triviality of the action.
Thus~\eqref{eq:kummer-cohomology} shows that $\sigma$ acts trivially on $H^1(\AbsGalGrp{L},\mathbb{Z}/n\mathbb{Z})$ if and only if
\begin{equation}
	\label{eq:kummer-trivial}
	\frac{\sigma(x)}{x}\in L^{\times n}\quad\text{for every }x\in L^\times.
\end{equation}

When $L$ does not contain $\mu_p$, we apply Kummer theory over $L(\mu_p)$ as follows.
Let $\Delta$ be a finite group of order prime to $p$, and let $\chi:\Delta\to\mathbb{Z}_p^\times$ be a character.
For a $\mathbb{Z}_p[\Delta]$-module $M$, write
\begin{equation*}
	M^\chi\coloneq\{m\in M:\delta m=\chi(\delta)m\text{ for every }\delta\in\Delta\}.
\end{equation*}
This is the $\chi$-eigensubmodule for the action of $\Delta$.
If $pM=0$, we read the scalars $\chi(\delta)$ modulo $p$.
The averaging operator
\begin{equation}
	\label{eq:chi-projector}
	\frac1{|\Delta|}\sum_{\delta\in\Delta}\chi(\delta)^{-1}\delta
\end{equation}
projects onto $M^\chi$.
Thus $M^\chi$ is a direct summand, taking $\chi$-components is exact, and reduction modulo $p$ gives a natural isomorphism
\begin{equation}
	\label{eq:chi-reduction}
	M^\chi/pM^\chi\xrightarrow{\sim}(M/pM)^\chi.
\end{equation}
Put $E\coloneq L(\mu_p)$ and $\Delta\coloneq\Gal(E/L)$, and let $\chi$ be the Teichm\"uller lift of the action of $\Delta$ on $\mu_p$.
Thus $\chi(\delta)$ is the unique $(p-1)$-st root of unity in $\mathbb{Z}_p^\times$ whose reduction gives the exponent by which $\delta$ acts on $\mu_p$.
Since $|\Delta|$ divides $p-1$, restriction--corestriction identifies $H^1(\AbsGalGrp{L},\mathbb{Z}/p\mathbb{Z})$ with $H^1(\AbsGalGrp{E},\mathbb{Z}/p\mathbb{Z})^\Delta$.
Under~\eqref{eq:kummer-cohomology}, applied over $E$, tensoring this invariant part with $\mu_p$ gives the $\chi$-component.
We therefore obtain a natural isomorphism
\begin{equation}
	\label{eq:kummer-descent}
	H^1(\AbsGalGrp{L},\mathbb{Z}/p\mathbb{Z})\otimes_{\mathbb{Z}/p\mathbb{Z}}\mu_p
	\xrightarrow{\sim}(E^\times/E^{\times p})^\chi.
\end{equation}
It is equivariant for automorphisms of $E$ that preserve $L$, acting on both tensor factors on the left.

\begin{lemma}
	\label{lem:prime-order-detection}
	Let $K$ be a discretely valued field of mixed characteristic $(0,p)$, and let $L/K$ be a cyclic extension of prime degree $\ell$.
	Assume that $\mu_p\subset K$.
	Then $\Gal(L/K)$ acts nontrivially on $H^1(\AbsGalGrp{L},\mathbb{Z}/p\mathbb{Z})$.
\end{lemma}

\begin{proof}
	Choose a normalized extension $v_L$ of the valuation to $L$.
	Let $\sigma$ generate $\Gal(L/K)$.
	Consider the group of norm-one principal units
	\begin{equation*}
		M\coloneq\{m\in U_L:N_{L/K}(m)=1\}.
	\end{equation*}
	We first show that $M$ is infinite.
	Choose $a\in L$ with $\sigma(a)\ne a$.
	For every $t\in K$ such that $v_L(t)+v_L(a)>0$ and $v_L(t)+v_L(\sigma(a))>0$, the element
	\begin{equation*}
		m_t\coloneq\frac{1+ta}{1+t\sigma(a)}
	\end{equation*}
	belongs to $U_L$, and its norm is $1$.
	Moreover, $m_t=m_s$ implies $(t-s)(a-\sigma(a))=0$, hence $t=s$.
	There are infinitely many such $t$, and hence $M$ is infinite.

	Next, assume that the action on $H^1(\AbsGalGrp{L},\mathbb{Z}/p\mathbb{Z})$ is trivial.
	We will show that $M$ is finite, contradicting the above result.
	Let $m\in M$.
	By Hilbert's theorem 90, $m$ has the form $\sigma(x)/x$ for some $x\in L^\times$.
	Since $\mu_p\subset K$, the automorphism $\sigma$ fixes $\mu_p$.
	The assumed trivial action and the criterion~\eqref{eq:kummer-trivial} give
	\begin{equation*}
		\sigma(x)/x\in L^{\times p}.
	\end{equation*}
	Hence $m=y^p$ for some $y\in L^\times$.
	Since $m\in U_L$, taking valuations gives
	\begin{equation*}
		p\,v_L(y)=v_L(y^p)=v_L(m)=0.
	\end{equation*}
	Thus $y\in\mathcal{O}_L^\times$, and reducing $y^p=m$ modulo $\mathfrak{m}_L$ gives $\bar y^p=\bar m=1$.
	Since the residue field has characteristic $p$, we have $(\bar y-1)^p=\bar y^p-1=0$.
	Hence $\bar y=1$, which gives $y\in U_L$ and $M\subseteq U_L^p$.

	By Lemma~\ref{lem:roots-of-unity}, choose a generator $\zeta$ of the finite cyclic group $\mu_{p^\infty}(L)$.
	By the same argument, every $p$-power root of unity in $L$ has residue $1$ and hence $\zeta\in U_L$.
	If $\zeta$ belonged to $M$, then the inclusion $M\subseteq U_L^p$ would imply that $L$ contains a $p$-th root of $\zeta$ of strictly larger $p$-power order.
	This contradicts the choice of $\zeta$.
	Hence $\zeta\in U_L\setminus M$.
	Therefore $N_{L/K}(\mu_{p^\infty}(L))$ is a nontrivial group of $p$-power roots of unity and contains $\mu_p$.
	The equality $m=y^p\in M$ gives $N_{L/K}(y)^p=1$.
	Since $N_{L/K}(y)\in\mu_p$, we can choose $\eta\in\mu_{p^\infty}(L)$ such that $N_{L/K}(\eta)=N_{L/K}(y)$.
	Since $y,\eta\in U_L$ and $N_{L/K}(y)=N_{L/K}(\eta)$, we have $y/\eta\in M$ and $m=\eta^p(y/\eta)^p$.
	It follows that $M\subseteq\mu_{p^\infty}(L)M^p$.

	Write $p^s$ for the order of $\zeta$.
	Raising the last inclusion to the $p^s$-th power removes the root-of-unity factor and gives
	\begin{equation*}
		M^{p^s}\subseteq M^{p^{s+1}}=(M^{p^s})^p.
	\end{equation*}
	If $M^{p^s}\ne\{1\}$, choose $b\in M^{p^s}\setminus\{1\}$ for which the positive integer $v_L(b-1)$ is minimal.
	The inclusion gives $b=c^p$ for some $c\in M^{p^s}$.
	Since $c\in U_L$ and the residue field has characteristic $p$, the sum $1+c+\cdots+c^{p-1}$ has residue zero.
	Hence
	\begin{equation}
		\label{eq:principal-unit-power-factorization}
		b-1=(c-1)(1+c+\cdots+c^{p-1})
	\end{equation}
	implies $v_L(b-1)>v_L(c-1)$, contradicting minimality.
	Hence $M^{p^s}=\{1\}$.
	Therefore, every element of $M$ is a root of $X^{p^s}-1$.
	Thus $M$ is finite, a contradiction.
\end{proof}

For complete fields, Proposition~\ref{prop:prime-order-detection} removes the assumption $\mu_p\subset K$.
Since Lemma~\ref{lem:prime-order-detection} already covers $p=2$, the remaining case has odd residue characteristic.
We first record a fact about $\mathbb{Z}_p$-linear actions.

\begin{lemma}
	\label{lem:finite-order-mod-p}
	Let $p$ be an odd prime, and let $M$ be a torsion-free $\mathbb{Z}_p$-module such that $\bigcap_{n\ge1}p^nM=0$.
	If a finite-order $\mathbb{Z}_p$-linear automorphism $T$ of $M$ acts trivially on $M/pM$, then $T=1$.
\end{lemma}

\begin{proof}
	If $T\ne1$, a power of $T$ has prime order $\ell$ and still acts trivially on $M/pM$.
	Replacing $T$ by this power, put $A\coloneq T-1$.
	Since $AM\subseteq pM$, we have $A^jM\subseteq p^{j-1}AM$ for $j\ge1$.
	Expanding $(1+A)^\ell=1$ gives
	\begin{equation*}
		\ell AM\subseteq\sum_{j=2}^\ell\binom{\ell}{j} p^{j-1}AM\subseteq \ell pAM.
	\end{equation*}
	For the last inclusion, if $\ell\ne p$, then $\ell$ is a $p$-adic unit and every summand has a factor $p$.
	If $\ell=p$, the terms with $2\le j<p$ have a factor $p^2$, and the last term has the factor $p^{p-1}$, which is divisible by $p^2$ since $p>2$.
	Since $M$ is torsion-free, we can cancel $\ell$ and obtain $AM\subseteq pAM$.
	Iteration gives $AM\subseteq p^nM$ for every $n\ge1$, hence $AM=0$.
	This contradicts the choice of $T$.
\end{proof}

For the application of Lemma~\ref{lem:finite-order-mod-p}, we use the following properties of logarithms.

\begin{lemma}
	\label{lem:principal-unit-logarithm}
	Let $E$ be a complete discretely valued field of mixed characteristic $(0,p)$.
	Then the usual logarithm is a homomorphism $\log:U_E\to E$ that commutes with valuation-preserving automorphisms.
	Its image is a $\mathbb{Z}_p$-submodule of $E$ satisfying
	\begin{equation*}
		p^2\mathcal{O}_E\subseteq\log U_E,
		\qquad \bigcap_{n\ge1}p^n\log U_E=0.
	\end{equation*}
\end{lemma}

\begin{proof}
	The logarithm is an equivariant homomorphism and restricts to an isomorphism
	\begin{equation*}
		\log:1+p^2\mathcal{O}_E\xrightarrow{\sim}p^2\mathcal{O}_E.
	\end{equation*}
	The power-series arguments in the proofs of \cite[Chapter~II, Propositions~5.4 and~5.5]{MR1697859} use completeness and estimates on the coefficients and apply to arbitrary residue fields.
	Thus $\log U_E$ contains $p^2\mathcal{O}_E$ and is open, hence closed, in $E$.
	Taking limits of integer multiples shows that it is a $\mathbb{Z}_p$-module.

	The factorization~\eqref{eq:principal-unit-power-factorization} gives $v_E(u^p-1)\ge v_E(u-1)+1$ for $u\in U_E$.
	Iteration gives $U_E^{p^r}\subseteq1+\mathfrak{m}_E^{r+1}$.
	Choose $r\ge0$ with $r+1\ge2v_E(p)$.
	Thus $U_E^{p^r}\subseteq1+p^2\mathcal{O}_E$ and $p^r\log U_E\subseteq p^2\mathcal{O}_E$.
	It follows that, for every $n\ge1$,
	\begin{equation*}
		p^n\log U_E\subseteq p^{n+2-r}\mathcal{O}_E.
	\end{equation*}
	Since $v_E(p)>0$, it follows that $\bigcap_{n\ge1}p^n\log U_E=0$.
\end{proof}

\begin{proposition}
	\label{prop:prime-order-detection}
	Let $K$ be a complete discretely valued field of mixed characteristic $(0,p)$, and let $L/K$ be a cyclic extension of prime degree $\ell$.
	Then $\Gal(L/K)$ acts nontrivially on $H^1(\AbsGalGrp{L},\mathbb{Z}/p\mathbb{Z})$.
\end{proposition}

\begin{proof}
	If $\mu_p\subset K$, apply Lemma~\ref{lem:prime-order-detection}.
	Assume therefore that $\mu_p\not\subset K$.
	Then $p$ is odd.
	Put $E\coloneq L(\mu_p)$ and $\Delta\coloneq\Gal(E/L)$.
	Let $\chi:\Gal(E/K)\to\mathbb{Z}_p^\times$ be the Teichm\"uller lift of the action on $\mu_p$.
	All $\chi$-components below are taken with respect to $\Delta$, using $\chi|_\Delta$.
	Let $\sigma$ generate $\Gal(L/K)$, and choose a lift $\tau\in\Gal(E/K)$ of $\sigma$.
	To cancel the action on the factor $\mu_p$ in~\eqref{eq:kummer-descent}, put
	\begin{equation*}
		M\coloneq(\log U_E)^\chi,
		\qquad T\coloneq\chi(\tau)^{-1}\tau.
	\end{equation*}
	Since $E/K$ is abelian, $T$ preserves $M$.
	Also $T^\ell=1$ on $M$, since $\tau^\ell\in\Delta$ acts there as multiplication by $\chi(\tau)^\ell$.

	We first show that $T\ne1$ on $M$.
	The averaging operator~\eqref{eq:chi-projector} on $E$ is $L$-linear and nonzero by linear independence of field automorphisms.
	Hence $E^\chi$ is a nonzero $L$-vector space.
	By Lemma~\ref{lem:principal-unit-logarithm}, the image $\log U_E$ contains $p^2\mathcal{O}_E$.
	Applying the averaging operator shows that $M$ spans $E^\chi$ over $\mathbb{Q}_p$.
	If $T=1$ on $M$, then $\tau$ acts on $E^\chi$ as multiplication by $\chi(\tau)$.
	For $0\ne b\in E^\chi$ and $x\in L$, we have $xb\in E^\chi$, and hence
	\begin{equation*}
		\sigma(x)=\tau(x)=\frac{\tau(xb)}{\tau(b)}
		=\frac{\chi(\tau)xb}{\chi(\tau)b}=x.
	\end{equation*}
	This contradicts the choice of $\sigma$.

	Next, assume that the action on $H^1(\AbsGalGrp{L},\mathbb{Z}/p\mathbb{Z})$ is trivial.
	The isomorphism~\eqref{eq:kummer-descent} shows that $T$ acts trivially on $(E^\times/E^{\times p})^\chi$.
	Inclusion of principal units and the logarithm induce $T$-equivariant maps
	\begin{equation*}
		(E^\times/E^{\times p})^\chi
		\hookleftarrow (U_E/U_E^p)^\chi\twoheadrightarrow M/pM.
	\end{equation*}
	The injection follows because a $p$-th root of a principal unit is again a principal unit.
	The surjection follows from $\log(u^p)=p\log u$, exactness of $\chi$-components, and~\eqref{eq:chi-reduction}.
	Hence $T$ acts trivially on $M/pM$.
	The module $M\subseteq E$ is torsion-free, and $\bigcap_{n\ge1}p^nM=0$ by Lemma~\ref{lem:principal-unit-logarithm}, since $M\subseteq\log U_E$.
	Lemma~\ref{lem:finite-order-mod-p} therefore gives $T=1$ on $M$, contradicting $T\ne1$.
\end{proof}

In the proof of Theorem~\ref{thm:mixed-characteristic}, we use the exponent-$p$ case of the Grunwald--Wang theorem \cite[Theorem~1]{MR2018557} to prove that the restriction map~\eqref{eq:mixed-local-restriction} is surjective.

\begin{theorem}
	\label{thm:mixed-characteristic}
	Let $K$ be a discretely valued field of mixed characteristic $(0,p)$.
	Then the conjugation action of $\Gal(L/K)$ on $\Abelianization{\AbsGalGrp{L}}/p$ is faithful for every finite Galois extension $L/K$.
	In particular, $\AbsGalGrp{K}$ is $p$-ab-faithful.
\end{theorem}

\begin{proof}
	We first prove that the action of $\Gal(L/K)$ on $H^1(\AbsGalGrp{L},\mathbb{Z}/p\mathbb{Z})$ is faithful.
	As in Lemma~\ref{lem:prime-order-test}, a nontrivial kernel would contain a subgroup of prime order.
	Replacing $K$ by its fixed field reduces the proof to showing nontriviality when $L/K$ is cyclic of prime degree $\ell$.
	Choose a normalized extension $v$ of the valuation to $L$, and let $\sigma$ generate $\Gal(L/K)$.
	Let $\Omega$ be the orbit of $v$ under $\Gal(L/K)$.
	Its distinct normalized valuations are pairwise inequivalent.
	The Grunwald--Wang theorem gives a $\Gal(L/K)$-equivariant surjection
	\begin{equation}
		\label{eq:mixed-local-restriction}
		\operatorname{res}:H^1(\AbsGalGrp{L},\mathbb{Z}/p\mathbb{Z})
		\twoheadrightarrow\prod_{w\in\Omega}H^1(\AbsGalGrp{\widehat L_w},\mathbb{Z}/p\mathbb{Z}).
	\end{equation}
	We show that $\sigma$ acts nontrivially on its target.

	Assume first that $\sigma$ preserves $v$.
	Then $\Omega=\{v\}$.
	Restriction identifies $\Gal(\widehat L_v/\widehat K)$ with $\Gal(L/K)$ by \cite[Chapter~II, Proposition~9.6]{MR1697859}.
	Proposition~\ref{prop:prime-order-detection} therefore gives a nontrivial action on $H^1(\AbsGalGrp{\widehat L_v},\mathbb{Z}/p\mathbb{Z})$.

	Assume next that $\sigma$ does not preserve $v$.
	Since $\ell$ is prime, $\sigma$ cyclically permutes the $\ell$ valuations in $\Omega$.
	Each factor of the target of~\eqref{eq:mixed-local-restriction} is nonzero.
	Indeed, put $F\coloneq\widehat L_w$.
	The degrees $[F(\mu_{p^n}):F]$ are unbounded: otherwise the union of these extensions would be a finite discretely valued extension containing all $p$-power roots of unity, contrary to Lemma~\ref{lem:roots-of-unity}.
	These extensions are abelian, and the prime-to-$p$ part of each degree divides $p-1$.
	Hence $\Gal(F(\mu_{p^n})/F)$ has a quotient of order $p$ for some $n$.
	Thus $H^1(\AbsGalGrp{F},\mathbb{Z}/p\mathbb{Z})\ne0$.
	An element of the product in~\eqref{eq:mixed-local-restriction} that is nonzero in exactly one factor is moved by $\sigma$.

	In both cases the action on the target is nontrivial.
	Since~\eqref{eq:mixed-local-restriction} is equivariant and surjective, the action on $H^1(\AbsGalGrp{L},\mathbb{Z}/p\mathbb{Z})$ is nontrivial.
	By the reduction to prime degree, the action is faithful for every finite Galois extension $L/K$.
	By~\eqref{eq:finite-cohomology-duality}, the action on $\Abelianization{\AbsGalGrp{L}}/p$ is faithful as well.
	Applying the cohomological condition of Lemma~\ref{lem:prime-order-test} with $\ell=p$ and $s=1$, also over finite extensions of $K$, proves that $\AbsGalGrp{K}$ is $p$-ab-faithful.
\end{proof}

\section{Ab-faithfulness for higher local fields}
\label{sec:higher-local-fields}

A \emph{higher local field} of dimension $n\geq0$ is a field $K=K_n$ with a residue tower
\begin{equation*}
	K_n,\ K_{n-1},\ \ldots,\ K_0
\end{equation*}
such that each $K_i$, for $i\geq1$, is complete discretely valued with $\resfield{K_i}=K_{i-1}$, and $K_0$ is finite.
In this section, we prove ab-faithfulness for higher local fields of positive dimension.
By Example~\ref{Iexampleemb}\ref{item:emb-positive-characteristic}, the absolute Galois group of every discretely valued field of positive characteristic is ab-faithful.
In characteristic zero, additional hypotheses are needed: the Newton--Puiseux theorem gives $\AbsGalGrp{\mathbb{C}((t))}\simeq\widehat{\mathbb{Z}}$, which is not ab-faithful.
The following lemma treats the case of residue characteristic zero.

\begin{lemma}
	\label{lem:residue-zero-lift}
	Let $K$ be a Henselian discretely valued field with $\operatorname{char}\resfield{K}=0$.
	Assume that $\AbsGalGrp{\resfield{K}}$ is ab-faithful.
	Then the following conditions are equivalent:
	\begin{enumerate}[(i)]
		\item\label{item:residue-zero-ab-faithful} $\AbsGalGrp{K}$ is ab-faithful.
		\item\label{item:residue-zero-slim} $\AbsGalGrp{K}$ is slim.
		\item\label{item:residue-zero-roots-of-unity} For every finite extension $\lambda/\resfield{K}$ and every prime $\ell$, the group $\mu_{\ell^\infty}(\lambda)$ is finite.
	\end{enumerate}
\end{lemma}

\begin{proof}
	The implication \ref{item:residue-zero-ab-faithful}$\Rightarrow$\ref{item:residue-zero-slim} follows from \cite[Remark~1.2]{Mochizuki2009CommentsStronglyTorsionFree}.
	We prove \ref{item:residue-zero-slim}$\Rightarrow$\ref{item:residue-zero-roots-of-unity}.
	Assume that $\AbsGalGrp{K}$ is slim.
	If $\mu_{\ell^\infty}(\lambda)$ is infinite for some finite extension $\lambda/\resfield{K}$ and some prime $\ell$, then $\lambda$ contains all $\ell$-power roots of unity.
	Let $L/K$ be the finite unramified extension with residue field $\lambda$.
	Unique Hensel lifting gives $\mu_{\ell^\infty}\subset L$.
	Since the residue characteristic is zero, the $\ell$-primary subgroup of the absolute inertia group of $L$ is isomorphic to $\mathbb{Z}_\ell(1)$, with conjugation given by the $\ell$-adic cyclotomic character; see \cite[Chapter~II, Corollary~9.15 and \S9, Exercise~1]{MR1697859}.
	Since the $\ell$-adic cyclotomic character of $\AbsGalGrp{L}$ is trivial, this nonzero inertia subgroup is contained in $Z(\AbsGalGrp{L})$.
	This contradicts the center-freeness of the open subgroup $\AbsGalGrp{L}$ of the slim group $\AbsGalGrp{K}$.

	It remains to prove \ref{item:residue-zero-roots-of-unity}$\Rightarrow$\ref{item:residue-zero-ab-faithful}.
	Assume that $\mu_{\ell^\infty}(\lambda)$ is finite for every finite extension $\lambda/\resfield{K}$ and every prime $\ell$.
	The assumptions hold for every finite extension of $K$: the residue extension is finite, open subgroups of an ab-faithful group are ab-faithful, and the finiteness assumption holds after further finite extensions.
	By Lemma~\ref{lem:prime-order-test} with $\Sigma$ the set of all primes, it suffices to show that the conjugation action of $\Gal(L/K)$ on $\Abelianization{\AbsGalGrp{L}}$ is nontrivial for every cyclic extension $L/K$ of prime degree $\ell$.
	Since $K$ is Henselian and $\operatorname{char}\resfield{K}=0$, the fundamental identity \cite[Chapter~II, Proposition~6.8]{MR1697859} gives
	\begin{equation*}
		[L:K]=e(L/K)[\resfield{L}:\resfield{K}]=\ell.
	\end{equation*}
	Thus $L/K$ is either unramified or totally ramified.
	In the unramified case, reduction identifies $\Gal(L/K)$ with $\Gal(\resfield{L}/\resfield{K})$.
	Under this identification, the residue Galois quotient induces an equivariant surjection
	\begin{equation*}
		\Abelianization{\AbsGalGrp{L}}\twoheadrightarrow \Abelianization{\AbsGalGrp{\resfield{L}}}
	\end{equation*}
	by \cite[Chapter~II, Propositions~9.9 and~9.11]{MR1697859}.
	The action on the target is nontrivial by the ab-faithfulness of $\AbsGalGrp{\resfield{K}}$.
	Therefore, the action on $\Abelianization{\AbsGalGrp{L}}$ is also nontrivial.

	Finally, we consider the totally ramified case, that is, $\resfield{L}=\resfield{K}$ and $v_L(K^\times)=\ell\mathbb{Z}$.
	Let $\pi$ be a uniformizer of $L$ and choose $a\in K$ with $v_L(a)=\ell$.
	Since the residue fields are equal, multiplying $a$ by a unit of $K$ makes $a/\pi^\ell\in U_L$.
	Since the residue characteristic is zero, $\ell$ is a unit in $\mathcal{O}_L$.
	Applying Hensel's lemma \cite[Section~12.5, (12.14)]{MR4647281} to $X^\ell-a/\pi^\ell$ at $1$ gives $u\in U_L$ with $u^\ell=a/\pi^\ell$.
	Replacing $\pi$ by $u\pi$, we may assume that $\pi^\ell=a\in K$.
	Let $\sigma$ generate $\Gal(L/K)$.
	Since $\pi^\ell\in K$ and $\pi\notin K$, the element $\sigma(\pi)/\pi$ is a primitive $\ell$-th root of unity.
	Reduction and unique Hensel lifting identify
	\begin{equation*}
		\mu_{\ell^\infty}(K)=\mu_{\ell^\infty}(L)
		\xrightarrow{\sim}\mu_{\ell^\infty}(\resfield{K}).
	\end{equation*}
	These groups are finite by assumption.
	Choose $s\ge1$ maximal with $\mu_{\ell^s}\subset L$.
	The element $\sigma(\pi)/\pi$ is not an $\ell^s$-th power in $L$, since any such root would have order $\ell^{s+1}$.
	Since $\sigma$ fixes $\mu_{\ell^s}\subset K$, the Kummer criterion~\eqref{eq:kummer-trivial} gives a nontrivial action on $H^1(\AbsGalGrp{L},\mathbb{Z}/\ell^s\mathbb{Z})$.
	By~\eqref{eq:finite-cohomology-duality}, the action on $\Abelianization{\AbsGalGrp{L}}$ is nontrivial.
\end{proof}

\begin{theorem}
	\label{thm:main-hlf}
	Let $K$ be a higher local field of positive dimension.
	Then $\AbsGalGrp{K}$ is ab-faithful.
\end{theorem}

\begin{proof}
	Finite extensions of a higher local field are again higher local fields of the same dimension.
	We first prove that $\mu_{\ell^\infty}(K)$ is finite for every prime $\ell$ by induction on the dimension $n$ of $K$.
	The case $n=0$ is clear since $K$ is finite.
	Assume that $n\ge1$.
	If $\ell\ne\operatorname{char}\resfield{K}$, reduction induces an injection
	\begin{equation*}
		\mu_{\ell^\infty}(K)\hookrightarrow\mu_{\ell^\infty}(\resfield{K}).
	\end{equation*}
	Indeed, if $\zeta\in U_K$ satisfies $\zeta^{\ell^r}=1$ for some $r\ge1$, then
	\begin{equation*}
		0=(\zeta-1)(1+\zeta+\cdots+\zeta^{\ell^r-1}).
	\end{equation*}
	The second factor has nonzero residue $\ell^r$, hence is a unit, and therefore $\zeta=1$.
	Finiteness follows by induction.
	If $\ell=\operatorname{char}\resfield{K}>0$, the group $\mu_{\ell^\infty}(K)$ is trivial when $\operatorname{char}K=\ell$, and is finite by Lemma~\ref{lem:roots-of-unity} when $\operatorname{char}K=0$.
	Thus $\mu_{\ell^\infty}(K)$ is finite for every prime $\ell$.

	Next, we prove by induction on the dimension of $K$ that $\AbsGalGrp{K}$ is ab-faithful.
	This follows from Example~\ref{Iexampleemb}\ref{item:emb-positive-characteristic} if $\operatorname{char}K>0$, and from Theorem~\ref{thm:mixed-characteristic} if $K$ has mixed characteristic.
	In the remaining case $\operatorname{char}\resfield{K}=0$, we have $n\ge2$, and $\AbsGalGrp{\resfield{K}}$ is ab-faithful by induction.
	Every finite extension of $\resfield{K}$ is a higher local field of dimension $n-1$.
	By induction, its group of $\ell$-power roots of unity is finite for every prime $\ell$.
	Hence the assertion follows from Lemma~\ref{lem:residue-zero-lift}.
\end{proof}

\section{Finite-step slimness and finite split embedding problems}
\label{sec:split-embedding-problems-slim}

Continuing Section~\ref{sec:split-embedding-problems}, we give a sufficient condition for slimness of finite-step solvable quotients in terms of finite split embedding problems with abelian kernel.
For a profinite group $G$, recall that
\begin{equation*}
	G^{[0]}\coloneq G,\qquad G^{[r+1]}\coloneq [G^{[r]},G^{[r]}],\qquad G^m\coloneq G/G^{[m]},
\end{equation*}
where all commutator subgroups are closed.
The following lemma gives the finite groups used in the proof of Proposition~\ref{prop:abelian-split-solvable-slim}.

\begin{lemma}
	\label{lem:faithful-abelian-extension-solvable-length}
	Let $m\ge2$ be an integer and let $A$ be a finite group.
	Assume that $A^{[m]}=1$.
	Then there exists a finite abelian group $C$ with a faithful $A$-action such that $(C\rtimes A)^{[m]}=1$.
\end{lemma}

\begin{proof}
	It suffices to construct a finite abelian group $C$ with a faithful $A$-action and an $A$-stable subgroup $W\subset C$ such that $A^{[m-1]}$ acts trivially on $W$ and $A$ acts trivially on $C/W$.
	Indeed, put $B\coloneq C\rtimes A$.
	Since $A$ acts trivially on $C/W$, we have $B/W\simeq(C/W)\times A$.
	Since $C/W$ is abelian and $m\ge2$, we have $(B/W)^{[m-1]}=\{0\}\times A^{[m-1]}$ under this identification.
	Taking the inverse image in $B$, we obtain
	\begin{equation*}
		B^{[m-1]}\subseteq W\rtimes A^{[m-1]}.
	\end{equation*}
	The group on the right is abelian: both factors are abelian, and $A^{[m-1]}$ acts trivially on $W$.
	Thus $B^{[m]}=1$.

	Next, we construct $C$ and $W\subset C$ with these properties.
	Put $M\coloneq A^{[m-1]}\times\mathbb{Z}/2\mathbb{Z}$.
	This is a nontrivial finite abelian group even when $A^{[m-1]}=1$.
	We write $M$ additively and identify $A^{[m-1]}$ with its first factor.
	Let $A$ act on $\operatorname{Map}(A,M)$ by $(a\cdot h)(x)\coloneq h(a^{-1}x)$ for every $a,x\in A$ and $h\in\operatorname{Map}(A,M)$.
	Via pullback along $A\to A/A^{[m-1]}$, we regard
	\begin{equation*}
		W\coloneq\operatorname{Map}(A/A^{[m-1]},M)
	\end{equation*}
	as an $A$-stable subgroup of $\operatorname{Map}(A,M)$.
	Since $M\ne0$, the action of $A/A^{[m-1]}$ on $W$ is faithful, and the kernel of the $A$-action on $W$ is exactly $A^{[m-1]}$.
	Choose a representative $s$ for each coset $sA^{[m-1]}$, with $1$ representing $A^{[m-1]}$.
	Then each element of $A$ has a unique expression $sd$ with $d\in A^{[m-1]}$.
	Define $f:A\to M$ by $f(sd)=d$, and put
	\begin{equation*}
		C\coloneq W+\mathbb{Z}f\subseteq\operatorname{Map}(A,M).
	\end{equation*}
	This is a finite abelian group.
	It remains to show that $C$ is stable under $A$, that $A$ acts trivially on $C/W$, and that $A$ acts faithfully on $C$.

	Since $A^{[m-1]}$ is abelian, for every $a,x\in A$ and $d\in A^{[m-1]}$, we have
	\begin{equation*}
		(a\cdot f-f)(xd)
		=f(a^{-1}x)+d-f(x)-d
		=(a\cdot f-f)(x).
	\end{equation*}
	Therefore $a\cdot f-f$ is constant on each coset of $A^{[m-1]}$, hence belongs to $W$.
	It follows that $a\cdot f\in f+W\subseteq C$, and therefore $C$ is $A$-stable.
	Moreover, $C/W$ is generated by $f+W$, which is fixed by every $a\in A$.
	Hence $A$ acts trivially on $C/W$.
	Finally, we show that the $A$-action on $C$ is faithful.
	If $a\in A$ acts trivially on $C$, then it acts trivially on $W$, hence $a\in A^{[m-1]}$.
	It also fixes $f$, and evaluation at $1$ gives $f(a^{-1})=f(1)=0$.
	Our choice of the representative $1$ implies that $f$ restricts to the inclusion $A^{[m-1]}\hookrightarrow M$.
	Thus $f(a^{-1})=0$ implies $a=1$.
	This proves faithfulness and completes the proof.
\end{proof}

\begin{proposition}
	\label{prop:abelian-split-solvable-slim}
	Let $G$ be a profinite group.
	Assume that every finite split embedding problem for $G$ with abelian kernel has a proper solution.
	Then $G^m$ is slim for every $m\ge2$.
\end{proposition}

\begin{proof}
	Fix $m\ge2$, let $H$ be an open subgroup of $G^m$, and assume that $1\ne z\in\operatorname{C}_{G^m}(H)$.
	Replacing $H$ by a smaller open normal subgroup of $G^m$, we may assume that $z\notin H$.
	By Lemma~\ref{lem:faithful-abelian-extension-solvable-length}, there exists a finite abelian group $C$ with a faithful $G^m/H$-action such that $B\coloneq C\rtimes(G^m/H)$ satisfies $B^{[m]}=1$.
	The hypothesis gives a surjection $G\twoheadrightarrow B$ over the natural map $G\twoheadrightarrow G^m/H$.
	Since $B^{[m]}=1$, this surjection factors through $G^m$ and gives a commutative diagram with exact rows and surjective vertical maps:
	\begin{equation*}
		\begin{array}{ccccccccc}
			1 & \to & H              & \to & G^m            & \to & G^m/H & \to & 1 \\
			  &     & \big\downarrow &     & \big\downarrow &     & \Vert           \\
			1 & \to & C              & \to & B              & \to & G^m/H & \to & 1
		\end{array}.
	\end{equation*}
	Since $z$ centralizes $H$, its image in $B$ centralizes $C$.
	Thus $zH\in G^m/H$ acts trivially on $C$, and faithfulness gives $zH=1$, contrary to $z\notin H$.
\end{proof}

If $K$ is a Hilbertian field, then every finite split embedding problem for $\AbsGalGrp{K}$ with abelian kernel has a proper solution by Ikeda's theorem \cite{MR121364} (see also~\cite[Proposition~18.4.5]{MR4647281}).
Hence Proposition~\ref{prop:abelian-split-solvable-slim} shows that $\AbsGalGrp{K}^m$ is slim for every $m\ge2$.

\section{Finite-step slimness and formal Hilbert~90}
\label{sec:solvable-centers}

We use formal Hilbert~90 and ab-faithfulness to give criteria for center-freeness and slimness.
For an abelian profinite group $B$, write $B_\ell$ for its $\ell$-primary component.
We write abelian profinite groups additively, but roots of unity multiplicatively.
For a set of primes $\Sigma$ and a profinite group $D$, put
\begin{equation*}
	R_\Sigma(D)\coloneq\ker\bigl(D\to\Abelianization{(D^\Sigma)}\bigr).
\end{equation*}
The universal properties of abelianization and maximal pro-$\Sigma$ quotients give a canonical isomorphism
\begin{equation}
	\label{eq:cf-sigma-primary}
	D/R_\Sigma(D)\simeq\Abelianization{(D^\Sigma)}\simeq\prod_{\ell\in\Sigma}(\Abelianization{D})_\ell.
\end{equation}
In particular, the subgroup $R_\Sigma(D)$ is characteristic in $D$.

\begin{definition}[{\cite[Section~2.2]{MR4880790}}]
	For a prime $\ell$, recall that an \emph{$\ell$-orientation} of a profinite group $G$ is a continuous homomorphism $\theta:G\to\mathbb{Z}_\ell^\times$.
	Write $T$ for $\mathbb{Z}_\ell$ with the action $g\cdot t=\theta(g)t$.
	The pair $(G,\theta)$ satisfies \emph{formal Hilbert~90} if reduction induces a surjection
	\begin{equation*}
		H^1(H,T/\ell^nT)\to H^1(H,T/\ell T)
	\end{equation*}
	for every open subgroup $H$ of $G$ and every $n\ge1$.
\end{definition}

By \cite[Lemma~2.3]{MR4880790}, this condition implies that reduction induces a surjection
\begin{equation}
	\label{eq:cf-cohomology-lifts}
	H^1(H,T/\ell^nT)\to H^1(H,T/\ell^mT)
\end{equation}
for every open subgroup $H$ of $G$ and every $n\ge m\ge1$.
We first prove a common criterion for the vanishing of invariant subgroups.

\begin{lemma}
	\label{lem:cf-integral-invariants}
	Let $G$ be a profinite group, let $\ell$ be a prime, and let $T$ be a free $\mathbb{Z}_\ell$-module of rank one with its $\ell$-adic topology and a continuous $G$-action given by a character $\theta:G\to\mathbb{Z}_\ell^\times$.
	Assume that $(G,\theta)$ satisfies formal Hilbert~90.
	Let $D$ be a closed normal subgroup of $G$ such that $D\subseteq G^{[1]}$, and put $A\coloneq(\Abelianization{D})_\ell$.
	Then the following hold:
	\begin{enumerate}[(1)]
		\item\label{item:cf-integral-global} We have $A^G=0$.
		\item\label{item:cf-integral-open} If $\theta(G)$ is infinite, then $A^H=0$ for every open subgroup $H$ of $G$.
	\end{enumerate}
\end{lemma}

\begin{proof}
	Let $z\in A$ be nonzero.
	Choose an integer $m\ge1$ and a continuous homomorphism $f_m:A\to T/\ell^mT$ such that $f_m(z)\ne0$.
	Since $D\subseteq G^{[1]}$, we have $\theta(D)=1$.
	Choose an open normal subgroup $H_1$ of $G$ such that $\theta(H_1)$ is trivial modulo $\ell^m$ and $D\cap H_1$ lies in the kernel of the composite $D\to A\xrightarrow{f_m}T/\ell^mT$.
	Put $H_0\coloneq DH_1$.
	Then $H_0$ acts trivially on $T/\ell^mT$, and the isomorphism $H_0/H_1\simeq D/(D\cap H_1)$ allows us to extend this composite to a continuous homomorphism $c_m:H_0\to T/\ell^mT$.

	By~\eqref{eq:cf-cohomology-lifts}, choose compatible classes $\alpha_n\in H^1(H_0,T/\ell^nT)$ for every $n\ge m$, starting with $\alpha_m=[c_m]$.
	Restriction to $D$ takes values in $H^1(D,T/\ell^nT)^{H_0}$.
	Since $D$ acts trivially on $T/\ell^nT$, we have $H^1(D,T/\ell^nT)=\Hom(A,T/\ell^nT)$.
	Thus these restrictions define compatible $H_0$-equivariant continuous homomorphisms $f_n:A\to T/\ell^nT$.
	Their inverse limit is a continuous homomorphism $f:A\to T$.
	Its reduction modulo $\ell^m$ is $f_m$, hence $f(z)\ne0$.
	For every $h\in H_0$ and $a\in A$, equivariance gives
	\begin{equation}
		\label{eq:cf-detecting-character}
		f(hah^{-1})=\theta(h)f(a).
	\end{equation}
	By~\eqref{eq:cf-detecting-character}, if $h\in H_0$ fixes $z$, then $(\theta(h)-1)f(z)=0$.
	Since $T$ is free of rank one and $f(z)\ne0$, it follows that $\theta(h)=1$.

	\noindent\ref{item:cf-integral-global} Assume that $z\in A^G$.
	The preceding argument gives $\theta|_{H_0}=1$.
	Thus $H^1(H_0,T/\ell^nT)=\Hom(H_0,T/\ell^nT)$, and the compatible classes $\alpha_n$ define a continuous homomorphism $c:H_0\to T$ whose restriction to $D$ induces $f$.
	Since $D\subseteq G^{[1]}$, the composite
	\begin{equation*}
		A\xrightarrow{0}(\Abelianization{G})_\ell
		\xrightarrow{\operatorname{Ver}_{G,H_0}}(\Abelianization{H_0})_\ell\to T
	\end{equation*}
	is zero, where the last map is induced by $c$.
	By the transfer formula \cite[Proposition~1.5.9]{MR2392026}, this composite sends $z$ to
	\begin{equation*}
		\sum_{g\in G/H_0}f(gzg^{-1})=[G:H_0]f(z).
	\end{equation*}
	This is nonzero because $T$ is torsion-free and $f(z)\ne0$, a contradiction.

	\noindent\ref{item:cf-integral-open} Assume that $\theta(G)$ is infinite and that $z\in A^H$ for an open subgroup $H$ of $G$.
	Every element of $H_0\cap H$ fixes $z$, hence $H_0\cap H\subseteq\ker\theta$.
	Thus $\ker\theta$ is open, contrary to the infinitude of $\theta(G)$.
\end{proof}

The remaining step is group-theoretic.
The following lemma locates the centralizers needed in the proof of Theorem~\ref{thm:cf-relative-criterion}.

\begin{lemma}
	\label{lem:cf-centralizer-criterion}
	Let $\Sigma$ be a set of primes, let $G$ be a $\Sigma$-ab-faithful profinite group, and let $D$ be a closed normal subgroup of $G$.
	Let $H$ be an open subgroup of $G$ that contains $D$.
	Put $R\coloneq R_\Sigma(D)$.
	Then
	\begin{equation*}
		\operatorname{C}_{G/R}(H/R)\subseteq D/R.
	\end{equation*}
\end{lemma}

\begin{proof}
	Let $g\in G$ have image in $\operatorname{C}_{G/R}(H/R)$.
	Assume that $g\notin D$.
	Choose an open normal subgroup $H_1$ of $G$ such that $H_1\subseteq H$ and $g\notin DH_1$.
	For every profinite group $J$ that contains $D$ as a closed subgroup, the composite $D\to J\to\Abelianization{(J^\Sigma)}$ factors through $\Abelianization{(D^\Sigma)}$.
	Hence
	\begin{equation}
		\label{eq:cf-sigma-monotone}
		R\subseteq R_\Sigma(J).
	\end{equation}
	By~\eqref{eq:cf-sigma-monotone}, the natural map $DH_1\to\Abelianization{((DH_1)^\Sigma)}$ factors through $DH_1/R$.
	Since $DH_1\subseteq H$, the element $g$ acts trivially on $\Abelianization{((DH_1)^\Sigma)}$.
	The $\Sigma$-ab-faithfulness of $G$ gives $g\in DH_1$, a contradiction.
	Thus $g\in D$, proving the assertion.
\end{proof}

\begin{theorem}
	\label{thm:cf-relative-criterion}
	Let $\Sigma$ be a set of primes, let $G$ be a $\Sigma$-ab-faithful profinite group, and let $D$ be a closed normal subgroup of $G$ such that $D\subseteq G^{[1]}$.
	Put $R\coloneq R_\Sigma(D)$.
	Assume that, for every prime $\ell\in\Sigma$ with $(\Abelianization{D})_\ell\ne0$, there exists an $\ell$-orientation $\theta_\ell$ of $G$ satisfying formal Hilbert~90.
	Then $G/R$ is center-free.
	If each $\theta_\ell$ can be chosen with infinite image, then $G/R$ is slim.
\end{theorem}

\begin{proof}
	We use the $G$-equivariant isomorphism~\eqref{eq:cf-sigma-primary}:
	\begin{equation*}
		D/R\simeq\prod_{\ell\in\Sigma}(\Abelianization{D})_\ell.
	\end{equation*}
	Lemma~\ref{lem:cf-integral-invariants}\ref{item:cf-integral-global}, applied to each nonzero factor with its orientation $\theta_\ell$, gives the vanishing of the $G$-invariant subgroup of each factor.
	Taking invariant subgroups in~\eqref{eq:cf-sigma-primary} therefore gives $(D/R)^{G}=0$.
	Lemma~\ref{lem:cf-centralizer-criterion}, applied with $H=G$, shows that the center of $G/R$ is contained in $D/R$.
	The vanishing of $(D/R)^{G}$ therefore proves the first assertion.

	Next, we prove slimness.
	Assume that each $\theta_\ell$ has infinite image.
	Lemma~\ref{lem:cf-integral-invariants}\ref{item:cf-integral-open} gives
	\begin{equation*}
		(D/R)^H\simeq\prod_{\ell\in\Sigma}((\Abelianization{D})_\ell)^H=0
	\end{equation*}
	for every open subgroup $H$ of $G$.
	Let $H$ be an open normal subgroup of $G$ that contains $R$.
	Both $\operatorname{C}_{G/R}(H/R)$ and $D/R$ are normal in $G/R$.
	Let $c\in\operatorname{C}_{G/R}(H/R)$ and $d\in D/R$.
	By normality, $dc^{-1}d^{-1}$ belongs to the first subgroup and $cdc^{-1}$ belongs to the second.
	Thus $c(dc^{-1}d^{-1})=(cdc^{-1})d^{-1}$ belongs to both subgroups, and
	\begin{equation*}
		cdc^{-1}d^{-1}\in\operatorname{C}_{G/R}(H/R)\cap(D/R)=(D/R)^H=0.
	\end{equation*}
	Hence $c$ and $d$ commute.
	It follows that
	\begin{equation*}
		\operatorname{C}_{G/R}(H/R)=\operatorname{C}_{G/R}(DH/R)\subseteq D/R,
	\end{equation*}
	where the inclusion follows from Lemma~\ref{lem:cf-centralizer-criterion} applied to the open subgroup $DH$, which contains $D$.
	Thus $\operatorname{C}_{G/R}(H/R)=0$.
	Every open subgroup of $G/R$ contains an open normal subgroup, proving slimness.
\end{proof}

\begin{corollary}
	\label{cor:cf-group-criterion}
	Let $G$ be an ab-faithful profinite group such that, for every prime $\ell$, there exists an $\ell$-orientation of $G$ satisfying formal Hilbert~90.
	Then $G^m$ is center-free for every $m\ge2$.
	If, for every prime $\ell$, such an orientation can be chosen with infinite image, then $G^m$ is slim for every $m\ge2$.
\end{corollary}

\begin{proof}
	Apply Theorem~\ref{thm:cf-relative-criterion} with $\Sigma$ equal to the set of all primes and $D\coloneq G^{[m-1]}$.
\end{proof}

If $G$ is ab-torsion-free, then every trivial orientation satisfies formal Hilbert~90.
Thus Corollary~\ref{cor:cf-group-criterion} generalizes the center-freeness assertion in \cite[Proposition~2.5]{yamaguchi2026centerfreenessfinitestepsolvablegroups}.

\begin{corollary}
	\label{thm:cf-absolute-center}
	Let $K$ be a field whose absolute Galois group $\AbsGalGrp{K}$ is ab-faithful.
	Then $\AbsGalGrp{K}^m$ is center-free for every $m\geq2$.
	Assume moreover that the following two conditions hold:
	\begin{enumerate}[(i)]
		\item\label{item:cf-roots-of-unity} For every prime $\ell\ne\operatorname{char}K$, the extension $K(\mu_{\ell^\infty})/K$ is infinite.
		\item\label{item:cf-artin-schreier} If $\operatorname{char}K=p>0$, then $K$ admits a cyclic Galois extension of degree $p$.
	\end{enumerate}
	Then $\AbsGalGrp{K}^m$ is slim for every $m\ge2$.
\end{corollary}

\begin{proof}
	By Corollary~\ref{cor:cf-group-criterion}, it suffices to construct, for every prime $\ell$, an $\ell$-orientation of $\AbsGalGrp{K}$ satisfying formal Hilbert~90, with infinite image under conditions~\ref{item:cf-roots-of-unity} and~\ref{item:cf-artin-schreier}.
	Fix a prime $\ell$.
	Assume first that $\ell\ne\operatorname{char}K$.
	Take $T=\mathbb{Z}_\ell(1)\coloneq\varprojlim_n\mu_{\ell^n}$, viewed as an additive $\mathbb{Z}_\ell$-module, with the cyclotomic action $\theta=\chi_\ell$.
	For every finite separable extension $L/K$ and every $n\ge m\ge1$, Kummer theory identifies the reduction map
	\begin{equation*}
		H^1(\AbsGalGrp{L},\mu_{\ell^n})\to H^1(\AbsGalGrp{L},\mu_{\ell^m})
	\end{equation*}
	with the natural surjection $L^\times/(L^\times)^{\ell^n}\to L^\times/(L^\times)^{\ell^m}$.
	Thus $(\AbsGalGrp{K},\chi_\ell)$ satisfies formal Hilbert~90.
	Moreover, $\chi_\ell(\AbsGalGrp{K})$ is infinite under condition~\ref{item:cf-roots-of-unity}.

	Assume next that $\ell=\operatorname{char}K=p>0$.
	We have $\operatorname{cd}_p(H)\le1$ for every open subgroup $H$ of $\AbsGalGrp{K}$ by \cite[Corollary~6.1.3]{MR2392026}.
	For any continuous character $\theta:\AbsGalGrp{K}\to\mathbb{Z}_p^\times$, give $T=\mathbb{Z}_p$ the action $g\cdot t=\theta(g)t$.
	For $n>m\ge1$, the exact sequence
	\begin{equation*}
		0\to T/p^{n-m}T\xrightarrow{p^m}T/p^nT\to T/p^mT\to0
	\end{equation*}
	gives~\eqref{eq:cf-cohomology-lifts}, since $H^2(H,T/p^{n-m}T)=0$.
	Hence every $p$-orientation satisfies formal Hilbert~90; taking $\theta=1$ gives the orientation required for center-freeness.

	Under condition~\ref{item:cf-artin-schreier}, a cyclic extension of degree $p$ gives a continuous surjection $\AbsGalGrp{K}\to\mathbb{Z}/p\mathbb{Z}$.
	With trivial actions on the coefficients, the cohomology groups in~\eqref{eq:cf-cohomology-lifts} are groups of continuous homomorphisms.
	Taking compatible lifts gives a continuous homomorphism $\rho:\AbsGalGrp{K}\to\mathbb{Z}_p$ whose reduction modulo $p$ is surjective.
	Its image is closed and is not contained in $p\mathbb{Z}_p$, hence $\rho$ is surjective.
	The continuous character
	\begin{equation*}
		\theta(g)\coloneq(1+p^2)^{\rho(g)}
	\end{equation*}
	has infinite image, since $a\mapsto(1+p^2)^a$ is injective on $\mathbb{Z}_p$ by the $p$-adic logarithm.
	As shown above, this orientation satisfies formal Hilbert~90, proving the assertion.
\end{proof}

\begin{example}
	\label{ex:cf-field-slim}
	We give applications of Corollary~\ref{thm:cf-absolute-center}.
	Condition~\ref{item:cf-roots-of-unity} in Corollary~\ref{thm:cf-absolute-center} is equivalent to the finiteness of $\mu_{\ell^\infty}(L)$ for every finite extension $L/K$ and every prime $\ell\ne\operatorname{char}K$.
	\begin{enumerate}[(1)]
		\item\label{item:cf-toral-kummer-faithful} A field $K$ of characteristic zero is \emph{torally Kummer-faithful} if, for every finite extension $L/K$, we have
		      \begin{equation*}
			      \bigcap_{n\geq1}(L^\times)^n=\{1\}\quad\text{(see~\cite[Definition~1.5]{MR3445958}).}
		      \end{equation*}
		      If $K$ is torally Kummer-faithful, then the image of $\chi_\ell$ is open for every prime $\ell$ by \cite[Remark~1.5.1]{MR3445958}.
		      Hence, if $K$ is torally Kummer-faithful and $\AbsGalGrp{K}$ is ab-faithful, Corollary~\ref{thm:cf-absolute-center} shows that $\AbsGalGrp{K}^m$ is slim for every $m\ge2$.

		\item A field $K$ is \emph{sub-$p$-adic} if it embeds into a finitely generated extension of $\mathbb{Q}_p$.
		      Every sub-$p$-adic field is torally Kummer-faithful by \cite[Remark~1.5.4(i)]{MR3445958}.
		      It also admits a discrete valuation of mixed characteristic $(0,p)$.
		      Hence $\AbsGalGrp{K}$ is ab-faithful by Theorem~\ref{thm:mixed-characteristic}, and $\AbsGalGrp{K}^m$ is slim for every $m\ge2$ by~\ref{item:cf-toral-kummer-faithful}.

		\item\label{item:cf-hlf-slim} Let $K$ be a higher local field of positive dimension.
		      Then $\AbsGalGrp{K}^m$ is slim for every $m\ge2$.
		      Indeed, Theorem~\ref{thm:main-hlf} and its proof give ab-faithfulness of $\AbsGalGrp{K}$ and finiteness of $\mu_{\ell^\infty}(L)$ for every finite extension $L/K$ and every prime $\ell$.
		      If $\operatorname{char}K=p>0$, the composite of the residue Galois quotients to the finite last residue field gives a surjection $\AbsGalGrp{K}\twoheadrightarrow\widehat{\mathbb{Z}}$.
		      Its quotient of order $p$ gives a cyclic Galois extension of $K$ of degree $p$.
		      Thus Corollary~\ref{thm:cf-absolute-center} applies.
	\end{enumerate}
\end{example}

The next lemma shows that ab-faithfulness in Corollary~\ref{thm:cf-absolute-center} cannot be replaced by slimness.

\begin{lemma}
	\label{ex:slim-perfect-kernel}
	There exists a pseudo-algebraically closed (PAC) field $K$ of characteristic zero such that $\AbsGalGrp{K}$ is slim and $\AbsGalGrp{K}^m\simeq\widehat{\mathbb{Z}}$ for every $m\ge1$.
\end{lemma}

\begin{proof}
	We first construct a projective profinite group with the required properties.
	Let $P\to A_5$ be the universal Frattini cover.
	The group $P$ is projective by \cite[Proposition~25.7.1]{MR4647281}.
	Since $A_5$ is perfect, $[P,P]$ maps onto $A_5$.
	A Frattini cover has no proper closed subgroup that maps onto its target, hence $[P,P]=P$.
	Put $G\coloneq P\amalg\widehat{\mathbb{Z}}$, where $\amalg$ denotes the free profinite product.
	Both factors are projective, since $\widehat{\mathbb{Z}}$ is free profinite of rank one.
	A finite embedding problem for $G$ has a weak solution: lift the homomorphisms on the two factors by projectivity and combine them by the universal property of the free profinite product.
	Thus $G$ is projective, and \cite[Corollary~26.1.2]{MR4647281}, with base field $\mathbb{Q}$, realizes $G$ as the absolute Galois group of a PAC field $K$ of characteristic zero.

	Next, we compute the finite-step solvable quotients of $G$.
	Let $\pi:G\twoheadrightarrow\widehat{\mathbb{Z}}$ be the retraction that kills $P$, and put $D\coloneq\ker\pi$.
	The group $D$ is the closed normal subgroup of $G$ generated by $P$.
	Each conjugate of $P$ is perfect, hence is contained in $[D,D]$.
	Therefore $[D,D]=D$.
	Since $G/D$ is abelian, we have $[G,G]\subseteq D=[D,D]\subseteq[G,G]$.
	It follows that $G^{[m]}=D$ and $G^m\simeq\widehat{\mathbb{Z}}$ for every $m\ge1$.

	Finally, we prove that $G$ is slim.
	Let $g\in G$ lie outside $\widehat{\mathbb{Z}}$.
	The theorem of Herfort--Ribes on intersections of conjugate free factors \cite[Theorem~9.1.12]{MR2599132} gives
	\begin{equation*}
		\widehat{\mathbb{Z}}\cap g\widehat{\mathbb{Z}}g^{-1}=1.
	\end{equation*}
	Thus $\operatorname{C}_G(z)=\widehat{\mathbb{Z}}$ for every nontrivial $z\in\widehat{\mathbb{Z}}$: an element outside this factor cannot centralize $z$, whereas the factor itself is abelian.
	Fix $1\ne a\in P$, which lies outside $\widehat{\mathbb{Z}}$, and a topological generator $\tau$ of $\widehat{\mathbb{Z}}$.
	For an open subgroup $H$ of $G$, both $H\cap\widehat{\mathbb{Z}}$ and $a^{-1}Ha\cap\widehat{\mathbb{Z}}$ are open in $\widehat{\mathbb{Z}}$.
	Hence there exists $n\ge1$ with $\tau^n,a\tau^na^{-1}\in H$.
	Both elements are nontrivial, and their centralizers give
	\begin{equation*}
		\operatorname{C}_G(H)\subseteq\widehat{\mathbb{Z}}\cap a\widehat{\mathbb{Z}}a^{-1}=1.
	\end{equation*}
	This proves slimness.
\end{proof}

\section*{Acknowledgements}
The author would like to thank Takahiro~Murotani and Akio~Tamagawa for their valuable comments.
This work was supported by the Japan Society for the Promotion of Science (JSPS) KAKENHI Grant Number 23KJ0881.

\printbibliography

@misc{yamaguchi2026centerfreenessfinitestepsolvablegroups,
      title={Center-freeness of finite-step solvable groups arising from anabelian geometry},
      author={Naganori Yamaguchi},
      year={2026},
      eprint={2601.07112},
      archivePrefix={arXiv},
      primaryClass={math.GR},
      url={https://arxiv.org/abs/2601.07112},
}

@book {MR2392026,
    AUTHOR = {Neukirch, J\"urgen and Schmidt, Alexander and Wingberg, Kay},
     TITLE = {Cohomology of number fields},
    SERIES = {Grundlehren der mathematischen Wissenschaften [Fundamental
              Principles of Mathematical Sciences]},
    VOLUME = {323},
   EDITION = {Second},
 PUBLISHER = {Springer-Verlag, Berlin},
      YEAR = {2008},
     PAGES = {xvi+825},
      ISBN = {978-3-540-37888-4},
   MRCLASS = {11R34 (11-02 11G45 11R23 11S20 11S25 11S31 12G05)},
  MRNUMBER = {2392026},
       DOI = {10.1007/978-3-540-37889-1},
       URL = {https://doi.org/10.1007/978-3-540-37889-1},
}

@article {MR4880790,
    AUTHOR = {Merkurjev, Alexander and Scavia, Federico},
     TITLE = {On the {M}assey vanishing conjecture and formal {H}ilbert 90},
   JOURNAL = {Proc. Lond. Math. Soc. (3)},
  FJOURNAL = {Proceedings of the London Mathematical Society. Third Series},
    VOLUME = {130},
      YEAR = {2025},
    NUMBER = {3},
     PAGES = {Paper No. e70036, 32},
      ISSN = {0024-6115,1460-244X},
   MRCLASS = {12G05 (20E18 20J06 55S30)},
  MRNUMBER = {4880790},
MRREVIEWER = {Zinovy\ Reichstein},
       DOI = {10.1112/plms.70036},
       URL = {https://doi.org/10.1112/plms.70036},
}

@article {MR121364,
    AUTHOR = {Ikeda, Masatoshi},
     TITLE = {Zur {E}xistenz eigentlicher galoisscher {K}\"orper beim
              {E}inbettungsproblem f\"ur galoissche {A}lgebren},
   JOURNAL = {Abh. Math. Sem. Univ. Hamburg},
  FJOURNAL = {Abhandlungen aus dem Mathematischen Seminar der Universit\"at
              Hamburg},
    VOLUME = {24},
      YEAR = {1960},
     PAGES = {126--131},
      ISSN = {0025-5858,1865-8784},
   MRCLASS = {12.00},
  MRNUMBER = {121364},
MRREVIEWER = {O.\ F. G. Schilling},
       DOI = {10.1007/BF02942025},
       URL = {https://doi.org/10.1007/BF02942025},
}

@article {MR5005799,
    AUTHOR = {Hyeon, Seung-Hyeon},
     TITLE = {The {$m$}-step solvable anabelian geometry of
              mixed-characteristic local fields},
   JOURNAL = {J. Lond. Math. Soc. (2)},
  FJOURNAL = {Journal of the London Mathematical Society. Second Series},
    VOLUME = {112},
      YEAR = {2025},
    NUMBER = {6},
     PAGES = {Paper No. e70402},
      ISSN = {0024-6107,1469-7750},
   MRCLASS = {11S20 (11F80 11S15 11S31)},
  MRNUMBER = {5005799},
       DOI = {10.1112/jlms.70402},
       URL = {https://doi.org/10.1112/jlms.70402},
}

@book {MR4647281,
    AUTHOR = {Fried, Michael David and Jarden, Moshe},
     TITLE = {Field arithmetic},
    SERIES = {Ergebnisse der Mathematik und ihrer Grenzgebiete. 3. Folge. A
              Series of Modern Surveys in Mathematics [Results in
              Mathematics and Related Areas. 3rd Series. A Series of Modern
              Surveys in Mathematics]},
    VOLUME = {11},
   EDITION = {Fourth},
      NOTE = {Revised by Moshe Jarden},
 PUBLISHER = {Springer, Cham},
      YEAR = {2023},
     PAGES = {xxxi+827},
      ISBN = {978-3-031-28019-1},
   MRCLASS = {12E30 (03B25 03C10 03C20 03C60 11Rxx)},
  MRNUMBER = {4647281},
       DOI = {10.1007/978-3-031-28020-7},
       URL = {https://doi.org/10.1007/978-3-031-28020-7},
}

@book {MR2599132,
    AUTHOR = {Ribes, Luis and Zalesskii, Pavel},
     TITLE = {Profinite groups},
    SERIES = {Ergebnisse der Mathematik und ihrer Grenzgebiete. 3. Folge. A
              Series of Modern Surveys in Mathematics [Results in
              Mathematics and Related Areas. 3rd Series. A Series of Modern
              Surveys in Mathematics]},
    VOLUME = {40},
   EDITION = {Second},
 PUBLISHER = {Springer-Verlag, Berlin},
      YEAR = {2010},
     PAGES = {xvi+464},
      ISBN = {978-3-642-01641-7},
   MRCLASS = {20E18},
  MRNUMBER = {2599132},
       DOI = {10.1007/978-3-642-01642-4},
       URL = {https://doi.org/10.1007/978-3-642-01642-4},
}

@book {MR1697859,
    AUTHOR = {Neukirch, J\"urgen},
     TITLE = {Algebraic number theory},
    SERIES = {Grundlehren der mathematischen Wissenschaften [Fundamental
              Principles of Mathematical Sciences]},
    VOLUME = {322},
      NOTE = {Translated from the 1992 German original and with a note by
              Norbert Schappacher,
              With a foreword by G. Harder},
 PUBLISHER = {Springer-Verlag, Berlin},
      YEAR = {1999},
     PAGES = {xviii+571},
      ISBN = {3-540-65399-6},
   MRCLASS = {11Rxx (11-02 11S15 11S31 14C40)},
  MRNUMBER = {1697859},
MRREVIEWER = {Cornelius\ Greither},
       DOI = {10.1007/978-3-662-03983-0},
       URL = {https://doi.org/10.1007/978-3-662-03983-0},
}

@unpublished{Mochizuki2009CommentsStronglyTorsionFree,
  author  = {Mochizuki, Shinichi},
  title   = {Comments on Strongly Torsion-Free Groups},
  note    = {Unpublished note, February 2009},
  year    = {2009},
  month   = feb,
  url     = {https://www.kurims.kyoto-u.ac.jp/~motizuki/The%20Algebraic%20and%20Anabelian%20Geometry%20of%20Configuration%20Spaces%20%28comments%20on%20strongly%20torsion-free%20groups%29.pdf},
  urldate = {2025-12-30},
}

@article {MR2793956,
    AUTHOR = {Bary-Soroker, Lior and Duy T\^an, Nguy\^e\~n},
     TITLE = {On {$p$}-embedding problems in characteristic {$p$}},
   JOURNAL = {J. Pure Appl. Algebra},
  FJOURNAL = {Journal of Pure and Applied Algebra},
    VOLUME = {215},
      YEAR = {2011},
    NUMBER = {10},
     PAGES = {2533--2537},
      ISSN = {0022-4049,1873-1376},
   MRCLASS = {12F12 (12E25)},
  MRNUMBER = {2793956},
MRREVIEWER = {Montserrat\ Vela},
       DOI = {10.1016/j.jpaa.2011.02.012},
       URL = {https://doi.org/10.1016/j.jpaa.2011.02.012},
}

@article {MR4434497,
    AUTHOR = {Minamide, Arata and Tsujimura, Shota},
     TITLE = {Anabelian group-theoretic properties of the absolute {G}alois
              groups of discrete valuation fields},
   JOURNAL = {J. Number Theory},
  FJOURNAL = {Journal of Number Theory},
    VOLUME = {239},
      YEAR = {2022},
     PAGES = {298--334},
      ISSN = {0022-314X,1096-1658},
   MRCLASS = {12F10 (12J25 14H30)},
  MRNUMBER = {4434497},
MRREVIEWER = {Martin\ Epkenhans},
       DOI = {10.1016/j.jnt.2021.12.006},
       URL = {https://doi.org/10.1016/j.jnt.2021.12.006},
}

@article {MR2677893,
    AUTHOR = {Bary-Soroker, Lior and Haran, Dan and Harbater, David},
     TITLE = {Permanence criteria for semi-free profinite groups},
   JOURNAL = {Math. Ann.},
  FJOURNAL = {Mathematische Annalen},
    VOLUME = {348},
      YEAR = {2010},
    NUMBER = {3},
     PAGES = {539--563},
      ISSN = {0025-5831,1432-1807},
   MRCLASS = {20E18 (12E30 12F10 14H30)},
  MRNUMBER = {2677893},
MRREVIEWER = {Hiroaki\ Nakamura},
       DOI = {10.1007/s00208-010-0484-8},
       URL = {https://doi.org/10.1007/s00208-010-0484-8},
}

@article {MR3445958,
    AUTHOR = {Mochizuki, Shinichi},
     TITLE = {Topics in absolute anabelian geometry {III}: global
              reconstruction algorithms},
   JOURNAL = {J. Math. Sci. Univ. Tokyo},
  FJOURNAL = {The University of Tokyo. Journal of Mathematical Sciences},
    VOLUME = {22},
      YEAR = {2015},
    NUMBER = {4},
     PAGES = {939--1156},
      ISSN = {1340-5705},
   MRCLASS = {14G32 (14H30)},
  MRNUMBER = {3445958},
MRREVIEWER = {Ariyan\ Javanpeykar},
}

@incollection {MR2018557,
    AUTHOR = {Lorenz, Falko and Roquette, Peter},
     TITLE = {The theorem of {G}runwald-{W}ang in the setting of valuation
              theory},
 BOOKTITLE = {Valuation theory and its applications, {V}ol. {II}
              ({S}askatoon, {SK}, 1999)},
    SERIES = {Fields Inst. Commun.},
    VOLUME = {33},
     PAGES = {175--212},
 PUBLISHER = {Amer. Math. Soc., Providence, RI},
      YEAR = {2003},
      ISBN = {0-8218-3206-9},
   MRCLASS = {12J20 (11R37)},
  MRNUMBER = {2018557},
MRREVIEWER = {Franz-Viktor\ Kuhlmann},
       DOI = {10.1090/fic/033/12},
       URL = {https://doi.org/10.1090/fic/033/12},
}

\end{document}